\documentclass{amsart}
\usepackage{amsfonts,amscd}

\newtheorem{theorem}{Theorem}[section]
\newtheorem{lemma}[theorem]{Lemma}
\newtheorem{corollary}[theorem]{Corollary}
\newtheorem{proposition}[theorem]{Proposition}
\theoremstyle{remark}

\theoremstyle{definition}
\newtheorem{definition}[theorem]{Definition}
\newtheorem{example}[theorem]{Example}
\numberwithin{equation}{section}
  \DeclareMathOperator{\Kdb}{{\mathbb K}}
\DeclareMathOperator{\Cdb}{{\mathbb C}}

\DeclareMathOperator{\Ndb}{{\mathbb N}}

 \makeatother

\begin{document}

\title{Morita equivalence of dual operator algebras}

\author[David P. Blecher]{David P. Blecher*}
\address{Department of Mathematics, University of Houston, Houston, TX
77204-3008}
\email{dblecher@math.uh.edu}
\author{Upasana Kashyap}
\address{Department of Mathematics, University of Houston,
Houston, TX  77204-3008} \email{upasana@math.uh.edu}
\date{\today}
\thanks{*Supported by grant DMS 0400731 from the National Science Foundation}


\begin{abstract}
We consider notions of Morita equivalence
appropriate to weak* closed algebras of Hilbert space operators.
We obtain new variants, appropriate to the
dual algebra setting, of the basic theory of
strong Morita equivalence, and new
nonselfadjoint analogues of aspects of
Rieffel's $W^*$-algebraic Morita equivalence.
\end{abstract}

\maketitle

\section{Introduction and notation}

By definition, an {\em operator algebra} is a subalgebra of
  $B(H)$, the bounded operators on a Hilbert space $H$, which
 is closed in the norm
topology.  It is a {\em dual algebra} if it is closed in
 the weak* topology
(also known as the $\sigma$-weak topology).  In
  \cite{BMP},  the first author, Muhly, and Paulsen
generalized Rieffel's strong Morita equivalence of $C^*$-algebras,
to general operator algebras.  At that time however, we were not
clear about how to generalize Rieffel's variant for $W^*$-algebras
\cite{RF1}, to dual operator algebras.  Recently, two approaches
have been suggested for this, in \cite{BM2} and \cite{Elef1,Elef2},
each of which reflect (different) important aspects of Rieffel's
$W^*$-algebraic Morita equivalence. For example, the notion
introduced in \cite{Elef1,Elef2} is equivalent to the very important
notion of (weak*) `stable isomorphism' \cite{EP}.   The fact
remains, however, that  neither approach seems
able to treat certain other very important examples, such as the
second dual of a strong Morita equivalence. In the present paper we
examine a framework, part of which was suggested at the end of
\cite{BM2},
which does include all examples hitherto considered,
and which represents an important and
natural framework for the Morita equivalence of dual algebras.
It is also one to which all the relevant parts of the earlier theory of 
strong Morita equivalence (from e.g.\ \cite{BMP,BMN})
transfers in a very clean manner, indeed which may in some sense be summarized 
as `just changing the tensor product involved'.  In addition, it
may be easier in some cases to check the criteria for our
variant of Morita equivalence. 
Since many of the ideas and proofs are extremely analogous to those
from our papers on related topics, principally \cite{BMP,DB2,BMN}
and to a lesser extent \cite{DB6,DB3,DB4,BM2}, we will be quite
brief in many of the proofs. That is, we assume that the reader is a
little familiar with these earlier ideas and proof techniques, and
will often merely indicate the modifications to weak* topologies. A
more detailed exposition will be presented in the second authors
Ph.\ D.\ thesis \cite{UK}, along with many other related results.

In Section 2, we develop some basic tensor product properties which
we shall need.  In Section 3, we define our variant of Morita equivalence, and
present some of its consequences.  Section 4 is centered on the
`weak linking algebra',
the key tool for dealing with most aspects
of Morita equivalence, and in Section 5 we prove that if $M$ and
$N$ are weak* Morita equivalent dual operator algebras, then the
von Neumann algebras generated by $M$ and $N$ are Morita equivalent
in Rieffel's $W^*$-algebraic sense.

Turning to notation, if $E, F$ are sets, then $EF$ will mean the
norm closure of the span of products $z w$ for $z \in E,w \in F$. We
will assume that the reader is familiar with basic notions from {\em
operator space theory}, as may be found in any of the current texts
on that subject, e.g.\ \cite{ER1}, and the application of this
theory to operator algebras, as may be found in e.g.\ \cite{BLM}.
We study operator algebras
 from  an operator space point of view.  Thus  an {\em abstract
 operator algebra} $A$ is an operator space and a
 Banach algebra, for which there exists
  a Hilbert space $H$ and a completely isometric
 homomorphism $ \pi : A \rightarrow  B(H)$.

 We will often abbreviate `weak*' to `$w^*$'.
A {\em dual operator algebra} is an operator algebra which is also a
dual operator space. By well
 known duality principles,
 any $w^*$-closed subalgebra of $B(H)$, is a
 dual operator
 algebra.  Conversely, it is known (see e.g.\
 \cite{BLM}), that for any dual
 operator algebra $M$,  there exists a Hilbert
 space $H$ and a $w^*$-continuous completely isometric homomorphism
 $\pi: M \rightarrow B(H)$. In this case, the range $\pi(M)$ is a
 $w^*$-closed subalgebra of $B(H)$, which we may identify with $M$
in every way.
 We take all dual operator algebras to be {\em unital}, that is
we assume they each
possess an identity of norm $1$.  Nondual operator algebras in this
 paper, in contrast, will usually be approximately unital, that is,
 they possess a contractive approximate identity (cai).

 For cardinals or sets $I, J$, we use the symbol $M_{I,J}(X)$ for the
 operator space of $I \times J$
 matrices over $X$, whose `finite submatrices' have uniformly
 bounded norm.  We write $\Kdb_{I,J}(X)$ for the norm closure of
 these `finite submatrices'.   Then $C^w_J(X) = M_{J,1}(X), R^w_J(X) =
 M_{1,J}(X)$, and $C_J(X) = \Kdb_{J,1}(X)$ and $R_J(X) =
 \Kdb_{1,J}(X)$.
We sometimes write $\mathbb{M}_I(X)$ for
$M_{I,I}(X)$.

A  {\em concrete left operator module} over an operator algebra $A$,
 is a subspace $X\subset B(H)$ such that $ \pi(A)X$ $\subset X$ for a
 completely contractive representation $ \pi : A \rightarrow B(H)$. An
 {\em abstract operator  A-module} is an operator space $X$ which is also
 an $A$-module, such that $X$ is completely isometrically isomorphic,
 via an $A$-module map, to a concrete operator $A$-module. Most of the
 interesting modules over operator algebras are operator modules, such
 as Hilbert $C^*$-modules.  Similarly for right modules, or
 bimodules.  
By $_{M} \mathcal{H}$, we will mean the category of completely
contractive normal Hilbert modules over a dual operator algebra $M$.
That is, elements of $_{M} \mathcal{H}$ are pairs $(H, \pi)$, where
$H$ is a (column) Hilbert space (see e.g.\ 1.2.23 in
\cite{BLM}), and $\pi : M \to B(H)$ is a
$w^{*}$-continuous unital completely contractive representation. We
shall call such a map $\pi$ a {\em normal representation} of $M$.
The module action is expressed through the equation $m \cdot \zeta =
\pi(m) \zeta$.  The morphisms are bounded linear transformation
between Hilbert spaces that intertwine the representations, i.e.\ if
$(H_{i}, \pi_{i})$, $i=1,2$, are objects of the category  $_{M}
\mathcal{H}$, then the space of morphisms is defined as:
$B_{M}(H_{1},H_{2})$ = $\{T \in B(H_{1},H_{2}) : T\pi_{1}(m)=
\pi_{2}(m) T $ for all $m \in M\}$.

A {\em concrete dual operator $M$-$N$-bimodule} is a $w^{*}$-closed
 subspace $X$ of $B(K,H)$ such that $\theta(M) X \pi(N)$ $\subset X$, where
 $\theta$ and $\pi$ are normal
 representations of $M$ and $N$ on $H$ and $K$ respectively. An
 {\em abstract
 dual operator $M$-$N$-bimodule} is defined to be a nondegenerate
  operator
 $M$-$N$-bimodule $X$, which is also a dual operator space, such that
 the module actions are separately weak* continuous.  Such
 spaces can be represented completely isometrically as concrete
 dual operator bimodules, and in fact this can be done
 under even weaker hypotheses
 (see e.g.\ \cite{BLM,BM2,ER}).  Similarly for
one-sided modules (the case $M$ or $N$ equals $\Cdb$).
We shall write $_{M}
 \mathcal{R}$ for the category of left dual operator modules over
 $M$. The morphisms in $_{M} \mathcal{R}$ are the $w^*$-continuous
 completely bounded $M$-module maps.
We use standard notation for module mapping spaces, e.g.\
$CB(X,N)_{N}$ (resp.\ $w^*CB(X,N)_N$) are the completely bounded
(resp.\ and $w^*$-continuous) right $N$-module maps $X \to N$.
Any  $H \in \, _{M} \mathcal{H}$ (with its column Hilbert space
structure) is a left dual operator $M$-module.

If $M$ is a  dual operator algebra, then the {\em maximal
$W^*$-cover} $W^*_{\rm max}(M)$ is a $W^*$-algebra containing $M$ as
a $w^*$-closed subalgebra, and which is generated by $M$
as a $W^*$-algebra, and which
has the universal property: any normal
representation $\pi : M \rightarrow B(H)$
extends uniquely to a (unital) normal $*$-representation
$\tilde{\pi} : W^*_{\rm max}(M) \rightarrow B(H)$ (see  \cite{BSo}).
A normal representation $\pi
: M \rightarrow B(H)$ of a dual operator algebra $M$, or the
associated space $H$ viewed as an $M$-module, will be called {\em
normal universal}, if any
other
normal  representation is unitarily equivalent to
the restriction of a `multiple' of $\pi$ to a reducing subspace (see
\cite{BSo}).

\begin{lemma} \label{hi}    A normal representation $\pi
: M \rightarrow B(H)$ of a dual operator algebra $M$ is normal
universal iff its extension $\tilde{\pi}$ to $W^*_{\rm max}(M)$ is
one-to-one.
\end{lemma}

\begin{proof}  The ($\Leftarrow$) direction is stated in \cite{BSo}.
Thus there does exist a normal universal $\pi$ whose extension
$\tilde{\pi}$ to $W^*_{\rm max}(M)$ is one-to-one.
 It is observed
in \cite{BSo} that any other normal universal representation
$\theta$ is
quasiequivalent to  $\pi$.  It follows that the extension
$\tilde{\theta}$ to $W^*_{\rm max}(M)$ is quasiequivalent to
 $\tilde{\pi}$, and it is easy to see from this that
 $\tilde{\theta}$ is one-to-one.
\end{proof}

\section{Some tensor products}

We begin by recalling the definition of the Haagerup
 tensor product. Suppose $X$ and $Y$ are two
 operator spaces. Define $\lVert u \rVert_n $ for $ u \in M_n(X \otimes Y)$
 as:
$$ {\lVert u \rVert}_n  =  \inf\  \lbrace \lVert a \rVert \lVert b
 \rVert \ : \  u =a \odot b, a \in M_{np}(X), b \in M_{pn} (Y), p \in
 \mathbb{N} \rbrace.$$
\noindent Here $ a \odot b $ stands for the $ n \times n $ matrix
whose
 $ i,j $ -entry is $ \sum_{k=1}^{p} {a_{ik} \otimes b_{kj}} $.
The algebraic tensor product $ X \otimes Y$ with this sequence of
 matrix norms is an operator space. The completion of this operator space in
 the above norm is called {\em Haagerup  tensor  product}, and is
 denoted by $ X \otimes_{h} Y $. The completion of an operator space is an
 operator space, hence $X \otimes_{h} Y $ is an operator space.

If $X$ and $Y$ are respectively right and left operator $A$-modules,
 then  the {\em module Haagerup tensor product} $X \otimes_{hA} Y$ is
 defined to be the quotient of   $X \otimes_{h} Y$ by the closure of the
 subspace spanned by terms of the form $xa \odot y$ $-$ $x \odot ay$,
 for $x \in X$, $y \in Y$, $a \in A$.
Let $X$ be a right and $Y$ be a left operator $A$-module where $A$ is
 an operator algebra. We say that a bilinear map $ \psi : X \times Y
 \rightarrow W$ is {\em balanced} if $ \psi (xa,y)= \psi(x,ay)     $ for all $
 x\in X$, $y \in Y$ and $a \in A$. It is well known that the module
 Haagerup tensor product linearizes balanced
 bilinear maps which are completely contractive (or completely bounded)
 in the sense of
 Christensen
 and Sinclair (see e.g.\ 1.5.4 in \cite{BLM}).

 If $X$ and $Y$ are two operator spaces,
 then the {\em extended Haagerup tensor product} $X \otimes_{eh} Y$
 may be defined to be the subspace
 of $(X^* \otimes_{h} Y^*)^{*}$ corresponding to the
 completely bounded bilinear
 maps from $X^* \times Y^* \to \mathbb{C}$ which are
 separately weak$^{*}$-continuous.  If $X$ and $Y$ are
 dual operator spaces, with preduals $X_*$ and $Y_*$,
 then this coincides with the
 $weak^{*}$ {\em Haagerup tensor product} defined earlier in \cite{BS},
 and indeed
 $X \otimes_{eh} Y = (X_* \otimes_{h} Y_*)^*$.  The {\em normal Haagerup tensor product}
 $X \otimes^{\sigma h} Y$ is the
 operator space dual of  $X_* \otimes_{eh} Y_*$.  The
 canonical maps are complete isometries
  $$X \otimes_h Y \to X \otimes_{eh} Y \to X \otimes^{\sigma h} Y
  .$$
See  \cite{ER2} for more details.

\begin{lemma} \label{ballhag}  For any dual
operator spaces $X$ and $Y$,
 ${\rm Ball}(X \otimes_{h} Y)$ is $w^{*}$-dense in
 ${\rm Ball}(X \otimes^{\sigma h} Y)$.
\end{lemma}

\begin{proof}
Let $x$ $\in$ ${\rm Ball}( X \otimes^{\sigma h} Y)$ $\setminus$
 $\overline{{\rm Ball}(X \otimes_{h} Y)}^{w^{*}}$. By the
 geometric Hahn-Banach theorem,
 there exists a  $\phi$ $\in$ $(X \otimes^{\sigma h} Y)_{*}$, and $t$
 $\in$
 $\mathbb{R}$, such that $Re\  \phi(x)$ $>$ $t$ $>$ $Re\  \phi(y)$ for all
 $y$ $\in$ ${\rm Ball}(X \otimes_{h} Y)$. We view $\phi$ as
  a map $\phi : X \otimes_{h} Y \to \mathbb{C}$ corresponding to a
 completely contractive map from $X \times Y \to \mathbb{C}$ which is
 separately $w^{*}$-continuous.  It follows that
 $Re\  \phi(x)$ $>$ $t$ $>$ $\lvert  \phi(y) \rvert$ for all $y$ $\in$
 ${\rm Ball}(X \otimes_{h} Y)$, which implies that
 $\lVert \phi \rVert $ $\leq$ $t$.
 Thus $\lvert Re\ \phi(x) \rvert$ $\leq$ $\lVert \phi \rVert$
 $\lVert x \rVert$ $\leq$ $t$,  which is a contradiction.
\end{proof}

\begin{lemma} \label{asso1}
The normal Haagerup tensor product is associative. That is, if
 $X$,$Y$,$Z$ are dual operator spaces then
 $(X \otimes^{\sigma h} Y) \otimes^{\sigma h}  Z
 = X \otimes^{\sigma h} (Y \otimes^{\sigma h}  Z)$ as
 dual operator spaces.
\end{lemma}

\begin{proof}
This follows by the definition of the normal Haagerup tensor product
and using
 associativity of the extended Haagerup tensor product (e.g.\ see
 \cite{ER2}).
\end{proof}

 We now turn to the module version of the normal Haagerup tensor
 product, and review some facts from \cite{EP}.
 Let $X$ be a right dual operator $M$-module and $Y$ be a left dual operator
 $M$-module.   Let $(X \otimes_{h M} Y)_{\sigma}^{*}$  denote
 the subspace of $(X \otimes_{h} Y)^{*}$ corresponding to the
 completely bounded
 bilinear maps from $\psi : X \times Y \to \mathbb{C}$ which are
 separately weak$^{*}$-continuous
 and $M$-balanced (that is, $\psi(xm,y) = \psi(x,my))$. Define the
 {\em module normal Haagerup tensor product}
 $X \otimes^{\sigma h}_{M} Y$ to be the operator space dual of
 $(X \otimes_{h M} Y)_{\sigma}^{*}$. Equivalently,
 $X \otimes^{\sigma h}_{M}Y$ is the
 quotient of $X \otimes^{\sigma h} Y$ by the weak$^{*}$-closure of the
 subspace spanned by terms of the form $xm \otimes y$ $-$ $x \otimes my$,
 for $x \in X$, $y \in Y$, $m \in M$. The module normal Haagerup tensor
 product linearizes completely contractive, separately
 weak$^{*}$-continuous,
 balanced bilinear maps (see  \cite[Proposition 2.2]{EP}).
The canonical map $X \to Y \to X \otimes^{\sigma h} Y$ is such a map.

\begin{lemma} \label{E}
Let $X_{1}, X_{2}, Y_{1}, Y_{2}$ be dual operator spaces. If $u :
X_{1} \to Y_{1}$ and $v : X_{2} \to Y_{2}$ are $w^{*}$-continuous,
completely bounded, linear maps, then the map $u \otimes v$ extends
to a well defined $w^{*}$-continuous, linear, completely bounded map
from $X_{1} \otimes^{\sigma h}  X_{2} \to Y_{1} \otimes^{\sigma h}
Y_{2}$, with $\Vert u \otimes v \rVert_{cb}$ $\leq$ $\Vert u
\rVert_{cb}$$\lVert v \rVert_{cb}$.
\end{lemma}

\begin{proof}
This follows by considering the preduals of the maps, and using
the functoriality of the
extended Haagerup tensor product \cite{ER2}.  \end{proof}

\begin{corollary} \label{F}
Let $N$ be a dual algebra, let $X_{1}$ and $Y_{1}$ be dual operator
 spaces which are
  right $N$-modules, and let $X_{2}$, $Y_{2}$ be dual
 operator spaces which are left $N$-modules. If $u : X_{1} \to X_{2}$ and
 $v : Y_{1} \to Y_{2}$ are completely bounded, $w^{*}$-continuous,
 $N$-module maps, then the map $u \otimes v$ extends to a well defined linear,
 $w^{*}$-continuous, completely bounded map from
 $X_{1} \otimes^{\sigma h}_{N} Y_{1}  \to X_{2} \otimes^{\sigma h}_{N} Y_{2}$,
 with $\lVert u \otimes v \rVert_{cb}$ $\leq$ $\lVert u \rVert_{cb}$ $\lVert v \rVert_{cb}$.
\end{corollary}

\begin{proof}
Lemma \ref{E} gives a
 $w^{*}$-continuous, completely bounded, linear map
 $X_{1} \otimes^{\sigma h} Y_{1}  \to X_{2} \otimes^{\sigma h} Y_{2}$
 taking $x \otimes y$ to $u(x) \otimes v(y)$.
 Composing this map with the $w^{*}$-continuous,
 quotient map $X_{2} \otimes^{\sigma h} Y_{2}  \to X_{2}
\otimes^{\sigma h}_{N} Y_{2}$, we obtain a $w^{*}$-continuous,
completely bounded map
 $X_{1} \otimes^{\sigma h} Y_{1}  \to X_{2} \otimes^{\sigma h}_{N} Y_{2}$.
 It is easy to see that the kernel of the last map contains all
 terms of form $xn \otimes_{N} y - x \otimes_{N} ny$, with
 $n \in N, x \in X_{1}, y \in Y_{1}$.  This gives a map
 $X_{1} \otimes^{\sigma h}_{N} Y_{1}  \to X_{2} \otimes^{\sigma h}_{N} Y_{2}$
 with the required properties.
\end{proof}

\begin{lemma} \label{D}
If $X$ is a  dual operator $M$-$N$-bimodule and if $Y$ is a
 dual operator $M$-$L$-bimodule,  then $X\otimes^{\sigma h}_{N} Y$ is a
 dual operator $M$-$L$-bimodule.
\end{lemma}

\begin{proof}
To show e.g.\ it is a left dual operator $M$-module, use the canonical maps
$$M \otimes_h (X\otimes^{\sigma h} Y) \to M \otimes^{\sigma h}
(X\otimes^{\sigma h} Y) \to (M \otimes^{\sigma h} X) \otimes^{\sigma h} Y
\to X\otimes^{\sigma h} Y .$$
It follows from 3.3.1 in \cite{BLM}, the fact that
 $X \otimes Y$ is a weak* dense $M$-submodule, and the
universal property of $\otimes^{\sigma h}$, that
 $X\otimes^{\sigma h} Y$ is an operator $M$-module.  Composing 
the map $M \otimes^{\sigma h}
(X\otimes^{\sigma h} Y) \to  X\otimes^{\sigma h} Y$ above with
the canonical map $M \times (X\otimes^{\sigma h} Y) \to 
M \otimes^{\sigma h}
(X\otimes^{\sigma h} Y)$, one sees the module action is separately 
weak* continuous (see also
Lemma 2.3 in \cite{EP}).   By
3.8.8 in \cite{BLM},  $X\otimes^{\sigma h}_{N} Y$ is a dual operator 
$M$-module.
  \end{proof}

There is clearly a canonical map $X \otimes_{hM} Y \to X
\otimes^{\sigma h}_{M} Y$, with respect to which:

\begin{corollary} \label{ballmod}
For any dual operator $M$-modules $X$ and $Y$, the image of ${\rm
Ball}(X \otimes_{hM} Y)$ is $w^{*}$-dense in
 ${\rm Ball}(X \otimes^{\sigma h}_{M} Y)$.
\end{corollary}

\begin{proof}
Consider the canonical $w^{*}$-continuous quotient map $q : X
 \otimes^{\sigma h} Y \to X \otimes^{\sigma h}_{M} Y$ as in
  \cite[Proposition 2.1]{EP}.
 If $z$ $\in$ $X \otimes^{\sigma h}_{M} Y$ with $\lVert z \rVert$ $<$ $1$,
 then there exists  $z'$ $\in$ $X \otimes^{\sigma h} Y$ with
 $\lVert z' \rVert$ $<$ 1 such that $q(z')$ = $z$. By the above Lemma,
  there
 exists a net $(z_{t})$ in ${\rm Ball}(X \otimes_{h} Y)$ such that
 $z_{t}$ $\buildrel w^* \over \to$ $z'$.  Then
 $q(z_{t})$ $\buildrel w^* \over \to$ $q(z')$ = $z$.
\end{proof}

\begin{lemma} \label{A}  For any dual operator $M$-modules $X$ and $Y$,
and $m,n \in \Ndb$, we have  $M_{mn}(X \otimes^{\sigma h}_{M} Y)$
$\cong$ $C_{m}(X) \otimes^{\sigma h}_{M} R_{n}(Y)$ completely
isometrically and weak* homeomorphically. This is also true with $m,
n$ replaced by arbitrary cardinals: $M_{IJ} (X \otimes^{\sigma
h}_{M} Y)$ $\cong$ $C_{I}(X) \otimes^{\sigma h}_{M} R_{J}(Y)$.
\end{lemma}

\begin{proof}
We just prove the case that $m,n \in \Ndb$, the other being similar
(or can be deduced easily from Proposition \ref{asso2}).
 First we claim that $M_{mn} (X \otimes^{\sigma h } Y)$ $\cong $
$C_{m}(X) \otimes^{\sigma h} R_{n}(Y)$. Using facts from \cite{ER2}
and basic operator space duality, the predual of the latter space is
\begin{eqnarray*}
C_{m}(X)_{*} \otimes_{eh} R_{n}(Y)_{*} & \cong &   (R_{m}
\otimes_{h} X_{*} ) \otimes_{eh} (Y_{*} \otimes_{h}
C_{n}) \\
&\cong&   (R_{m} \otimes_{eh} X_{*}) \otimes_{eh} (Y_{*}
\otimes_{eh}
C_{n}) \\
 & \cong&   R_{m} \otimes_{eh} (X_{*}  \otimes_{eh}  Y_{*} )\otimes_{eh}
  C_{n} \\
  & \cong&   R_{m} \otimes_{h} (X_{*}  \otimes_{eh}  Y_{*} )\otimes_{h}
  C_{n} \\
& \cong& (X_{*} \otimes_{eh} Y_{*}) \overset{\frown}{\otimes}
(M_{mn})_* .
\end{eqnarray*}
We have used for example 1.5.14  in \cite{BLM}, 5.15 in \cite{ER2},
and associativity of the extended Haagerup
 tensor product \cite{ER2}.
The latter space is the predual of $M_{mn} (X \otimes^{\sigma h }
 Y)$, by e.g.\ 1.6.2 in \cite{BLM}.  This gives the claim.
 If $\theta$ is  the ensuing completely isometric isomorphism
 $C_{m}(X)   \otimes^{\sigma h} R_{n}(Y) \to M_{mn}(X \otimes^{\sigma h}
 Y )$, it is easy to check that $\theta$
 takes $[x_{1}\  x_{2}\       \hdots x_{m}]^{T}$  $\otimes$  $[y_{1} \ y_{2}  \hdots  y_{n}]$
 to the matrix $[x_{i} \otimes y_{j}]$.
Now  $C_{m}(X) \otimes^{\sigma  h}_{M} R_{n}(Y)$  = $C_{m}(X)
\otimes^{\sigma h} R_{n}(Y)/N$ where $N $ = $[xt \otimes y - x
\otimes ty]^{-w^{*}} $ with
 $x\in C_{m}(X), y \in R_{n}(Y), t \in M $.
 Let  $N'$ = $[x t \otimes y - x \otimes ty]^{-w^{*}}$
 where $x \in X, y \in Y, t \in M $, then clearly $\theta(N) = M_{mn}(N')$.
 Hence $$C_{m}(X) \otimes^{\sigma h} R_{n}(Y)/N \cong
 M_{mn}(X \otimes^{\sigma h} Y) /\theta(N) =
 M_{mn}(X \otimes^{\sigma h} Y)  /M_{nm}(N'),$$
which in turn equals $M_{mn}(X \otimes^{\sigma h} Y/N') = M_{mn}(X
\otimes^{\sigma h}_{M} Y)$.
\end{proof}

\begin{corollary} \label{ballmatrix} For any dual operator $M$-modules
$X$ and $Y$, and $m,n \in \Ndb$, we have that ${\rm Ball}(M_{mn}(X
\otimes_{h M}  Y)) $ is $w^{*}$-dense in {\rm Ball} $(M_{mn}(X
\otimes^{\sigma h}_{M} Y))$.
\end{corollary}

\begin{proof}
If $\eta$ $\in$ {\rm Ball} $(M_{mn}(X \otimes^{\sigma h}_{M}  Y))$,
then by Lemma \ref{A}, $\eta$ corresponds to  an element $\eta'$
$\in$ $C_{m}(X) \otimes ^{\sigma h}_{M} R_{n}(Y)$. By Lemma
\ref{ballmod}, there exists a net $(u_{t})$ in
 $C_{m}(X) \otimes_{h M} R_{n}(Y)$ such that
 $u_{t}   \buildrel w^* \over  \to \eta'$.
 By 3.4.11 in \cite{BLM}, $u_{t}$ corresponds to
 $u_{t}'$ $\in$ {\rm Ball}$(M_{mn}(X \otimes_{hM} Y))$
 such that $u_{t}' \buildrel w^{*} \over  \to \eta$.
\end{proof}

\begin{proposition} \label{asso2}
The normal module Haagerup tensor product is associative. That is,
if
 $M$ and $N$ are dual operator algebras, if $X$ is a right dual operator
 $M$-module, if $Y$ is a $M$-$N$-dual operator bimodule, and $Z$ is a
 left dual operator $N$-module, then
 $(X \otimes^{\sigma h}_{M} Y) \otimes^{\sigma h}_{N}  Z$
 is completely isometrically isomorphic to
 $X \otimes^{\sigma h}_{M} (Y \otimes^{\sigma h}_{N}  Z)$.
\end{proposition}

\begin{proof}
We define $X \otimes^{\sigma h}_{M} Y  \otimes^{\sigma h}_{N}  Z$ to
be
 the quotient of $X \otimes^{\sigma h} Y \otimes^{\sigma h} Z$ by the
 $w^{*}$-closure of the linear span of terms of the form
 $xm \otimes y \otimes z- x \otimes my \otimes z$ and
 $x \otimes yn \otimes z- x \otimes y \otimes nz$ with
 $x\in X, y \in Y, z \in Z, m \in M, n \in N$. By
 extending the arguments of Proposition 2.2 in \cite{EP} to the
 threefold normal module
 Haagerup tensor product, one sees that
 $X \otimes^{\sigma h}_{M} Y \otimes^{\sigma h}_{N}  Z$
 has the following universal property: If $W$
 is a dual operator space and $u : X \times Y \times Z \to W $ is a
 separately $w^{*}$-continuous, completely contractive, balanced, trilinear map,
 then there exists a $w^{*}$-continuous and completely contractive, linear map
 $\tilde{u} :X \otimes^{\sigma h}_{M} Y  \otimes^{\sigma h}_{N} Z \to W  $
 such that $\tilde{u}(x\otimes_{M} y \otimes_{N} z) = u(x,y,z)$.
     We will prove that  $(X \otimes^{\sigma h}_{M} Y ) \otimes^{\sigma h}_{N}  Z$
 has the above universal property defining
 $X \otimes^{\sigma h}_{M} Y  \otimes^{\sigma h}_{N}  Z$.
 Let $u : X \times Y \times Z \to W $ be a
 separately $w^{*}$-continuous, completely contractive,
 balanced, trilinear map.
 For each fixed $z \in Z$,  define $u_{z} : X \times Y \to W $
 by $u_{z}(x, y) = u(x,y,z)$.  This is  a separately
 $w^{*}$-continuous, balanced, bilinear map, which is completely
 bounded.
 Hence we obtain a $w^{*}$-continuous completely bounded linear map
 $u_{z}' : X \otimes^{\sigma h}_{M} Y \to W$ such that
 $u_{z}'(x \otimes_{M} y) = u_{z}(x,y)$.
 Define $u' :  (X \otimes^{\sigma h}_{M} Y) \times Z \to  W$ by
 $u'(a,z)$ =  $u_{z}'(a)$, for $a \in X \otimes^{\sigma h}_{M} Y$.
 Then $u'(x \otimes_{M} y,z) = u(x,y,z)$, and it
 is routine to check that $u'$ is bilinear and balanced over $N$.
 We will show that $u'$ is
completely contractive on  $(X \otimes_{ h M} Y) \times Z$, and then
the complete
 contractivity of $u'$ follows from Corollary \ref{ballmatrix}.
 Let $a$ $\in$ $M_{nm}(X \otimes_{hM} Y)$ with
 $\lVert a \lVert$ $<$ 1 and $z$ $\in$ $M_{mn}(Z)$ with
 $\lVert z \rVert$ $<$ 1. We want to show $\lVert u_{n}' (a,z) \rVert$ $<$ 1.
 It is well known that we can write $a = x \odot_{M} y$ where
 $x$ $\in$ $M_{nk}(X)$ and $y \in M_{km} (Y)$ for some
 $k \in \mathbb{N}$, with
 $\lVert x \rVert$ $<$ 1
 and $\Vert y \rVert$ $<$ 1. Hence $\lVert u_{n}' (a,z) \rVert$ =
 $\lVert u_{n}(x,y,z)\rVert$ $\leq$ $\lVert x \rVert  \lVert y \rVert \lVert z \rVert$ $<$ 1,
 proving $u'$ is completely contractive. By Proposition
 2.2 in \cite{EP}, we obtain a $w^{*}$-continuous, completely contractive,
 linear map $\tilde{u} :(X \otimes^{\sigma h}_{M} Y) \otimes^{\sigma h}_{N} Z  \to W  $
 such that $\tilde{u}((x \otimes_{M} y) \otimes_{N} z )$ =
 $u'(x \otimes_{M} y, z) $ = $u(x,y,z)$. This shows that
 $(X \otimes^{\sigma h}_{M} Y) \otimes^{\sigma h}_{N}  Z$ has the defining universal
 property of $X \otimes^{\sigma h}_{M} Y \otimes^{\sigma h}_{N}  Z$.
  Therefore $(X \otimes^{\sigma h}_{M} Y) \otimes^{\sigma h}_{N}  Z$ is
 completely isometrically isomorphic and $w^*$-homeomorphic
 to $X \otimes^{\sigma h}_{M} Y \otimes^{\sigma h}_{N}  Z$. Similarly
 $X \otimes^{\sigma h}_{M} (Y \otimes^{\sigma h}_{N}  Z) =
 X \otimes^{\sigma h}_{M} Y \otimes^{\sigma h}_{N} Z$.
\end{proof}

\begin{lemma} \label{C}
If $X$ is a left dual operator $M$-module then $M \otimes^{\sigma
h}_{ M} X$ is completely isometrically isomorphic to $X$.
\end{lemma}

\begin{proof}
As in Lemma 3.4.6 in \cite{BLM}, or follows from the
universal property.
\end{proof}

 \section{Morita contexts}

We now define two
 variants of Morita equivalence for unital dual operator
algebras, the first being more general than the second.
There are many equivalent variants of these definitions,
some of which we shall see later.

Throughout this section, we fix a pair of unital dual operator
algebras, $M$ and $N$, and a pair of dual operator
 bimodules $X$ and $Y$; $X$ will always be a $M$-$N$-bimodule and
 $Y$ will always be an $N$-$M$-bimodule.

\begin{definition} \label{w*Mor}  We say that $M$ is {\em weak*
Morita equivalent} to $N$, if  $M \cong X \otimes^{\sigma h}_{N} Y$
as dual operator $M$-bimodules (that is, completely
isometrically, $w^{*}$-homeomorphically, and also as
$M$-bimodules), and similarly if
 $N \cong Y \otimes^{\sigma h}_{M} X$
as dual operator $N$-bimodules.
We call $(M, N , X, Y)$ a {\em weak* Morita context} in this case.
\end{definition}

In this section, we will also fix  separately
 weak$^{*}$-continuous completely contractive bilinear maps
 $(\cdot,\cdot ) : X \times Y
 \to M$, and
 $\lbrack \cdot, \cdot \rbrack : Y \times X \to N$, and we will work 
with the  $6$-tuple, or {\em context}, $(M, N , X, Y, ( \cdot, \cdot ), \lbrack \cdot,
 \cdot \rbrack)$.

\begin{definition} \label{basicdef}  We say that $M$ is
{\em weakly Morita equivalent} to $N$, if there exist $w^*$-dense
operator algebras $A$ and $B$ in $M$ and $N$
respectively, and there exists a $w^*$-dense
 operator $A$-$B$-submodule $X^{'}$ in $X$,  and a $w^*$-dense
 $B$-$A$-submodule $Y^{'}$ in $Y$, such that the `subcontext'
 $(A,B,X^{'},Y^{'},(\cdot, \cdot ), \lbrack \cdot, \cdot \rbrack)$
 is a (strong) Morita context in the sense of \cite[Definition 3.1]{BMP}.
 In this case, we call $(M, N , X, Y)$ (or more properly the $6$-tuple above the
 definition),  a {\em weak Morita context}.
  \end{definition}

 $\mathbf{Remark.}$  Some authors use the term `weak Morita equivalence'
 for a quite different notion, namely to mean
 that the algebras have equivalent categories of Hilbert space
 representations.

\medskip

Weak Morita equivalence, as we have just defined it, is really
nothing more than the `weak$^{*}$-closure of' a strong Morita
equivalence in the sense of \cite{BMP}.  This definition includes
all examples that have hitherto been considered in the literature:

\medskip

 $\mathbf{Examples}$:

\begin{enumerate}
\item  We shall see in Corollary  \ref{isws} that every weak Morita
equivalence is an example of weak* Morita equivalence.
 \item   We shall see in Section 4 that every weak Morita
equivalence arises as follows:  Let $A, B$ be subalgebras
 of $B(H)$ and $B(K)$ respectively, for Hilbert spaces $H, K$, and
let $X \subset B(K,H), Y \subset B(H,K)$, such that the associated
subset ${\mathcal L}$ of $B(H \oplus K)$ is a subalgebra of $B(H
\oplus K)$, for Hilbert spaces $H, K$.  This is the same as
specifying a list of obvious algebraic conditions, such as $X Y
\subset A$.  Assume in addition that
$A$ possesses a cai $(e_t)$ with terms of the form $x y$, for $x \in
{\rm Ball}(R_n(X))$ and $y \in {\rm Ball}(C_n(Y))$, and $B$
possessing a cai with terms of a similar form $y x$ (dictated by
symmetry). Taking the weak* (that is, $\sigma$-weak) closure of all
these spaces clearly yields a weak Morita equivalence of
$\overline{A}^{w*}$ and $\overline{B}^{w*}$.
\item  Every weak* Morita
equivalence arises similarly to the setting in (2).
The main difference is that $A$, $B$ are unital, and $(e_t)$ is not
a cai, but $e_t \to 1_A$ weak*, and similarly for the net in $B$.
\item   Von
Neumann algebras which are Morita equivalent in Rieffel's
$W^*$-algebraic sense from \cite{RF1}, are clearly weakly Morita
equivalent.  We state this in the language of TROs. We recall that a
TRO is a subspace $Z \subset B(K,H)$ with $Z Z^* Z \subset Z$.
Rieffel's $W^*$-algebraic Morita equivalence of $W^*$-algebras $M$
and $N$ is essentially the same (see e.g.\ \cite[Section 8.5]{BLM}
for more details) as having a weak* closed TRO (that is, a {\em
WTRO}) $Z$, with $Z Z^*$ weak* dense in $M$ and $Z^* Z$ weak* dense
in $N$. Recall that $Z^* Z$ denotes the norm closure of the span of
products $z^* w$ for $z,w \in Z$. Here $(Z Z^*, Z Z^*, Z, Z^*)$ is
the weak* dense subcontext.

\item  More generally, the `tight Morita $w^*$-equivalence' of
\cite[Section 5]{BM2}, is easily seen to be a special case of weak
Morita equivalence.  In this  case, the  equivalence bimodules $X$
and $Y$
 are `selfdual'.  Indeed, this selduality is the great advantage
 of the approach of \cite[Section 5]{BM2}.

\item The second duals of strongly Morita equivalent operator algebras
 are weakly Morita equivalent. Recall that if $A$ and $B$ are
 approximately unital
 operator algebras, then $A^{**}$ and $B^{**}$ are unital dual operator
 algebras, by 2.5.6 in \cite{BLM}. If $X$ is a non-degenerate operator
 $A$-$B$-bimodule, then $X^{**}$ is a dual operator $A^{**}$-$B^{**}$-bimodule
 in a canonical way. Let $( \cdot, \cdot )$ be a bilinear map from
 $X \times Y$ to $A$ that is balanced over $B$ and is an $A$-bimodule map. Then
 notice that by 1.6.7 in \cite{BLM}, there is a unique separately $w^{*}$-continuous
 extension from $X^{**} \times Y^{**}$ to $A^{**}$, which we still call
 $( \cdot, \cdot )$.
 Now the weak Morita
 equivalence
 follows easily from the Goldstine lemma.

\item Any unital dual operator algebra $M$ is
weakly Morita
 equivalent to $\mathbb{M}_{I}(M)$, for any cardinal $I$. The
 weak* dense strong Morita subcontext in this case is
 $(M,\mathbb{K}_{I}(M),R_{I}(M),C_{I}(M))$,
whereas the equivalence bimodules $X$ and $Y$ above
 are $R_I^w(M)$ and $C_I^w(M)$ respectively.

 \item  {\em TRO equivalent} dual operator algebras $M$ and $N$,
or more generally {\em $\Delta$-equivalent algebras}, in the sense of
 \cite{Elef1,Elef2}, are weakly Morita
 equivalent.  If $M \subset B(H)$ and
$N \subset B(K)$, then TRO equivalence means
 that there
 exists a TRO $Z$ $\subset$ $B(H,K)$ such that $M$ = $[Z^{*} N Z
 ]^{\overline w^{*}}$ and $N$ = $[Z M Z^{*}]^{\overline w^{*}}.$
Eleftherakis shows that one may assume that $Z$ is a WTRO and
$1_N z = z 1_M = z$ for $z \in Z$.
Define $X$ and $Y$ to be the weak*  closures of $MZ^{*}N$
and $N Z M$ respectively.
Define $A$ and $B$ to be,
respectively,  $Z^* N Z$ and $Z M Z^*$.
Define $X'$ and $Y'$ to be,
respectively, the norm closures of $Z^* Y Z^*$ and $Z X Z$.
Since $Z$ is a TRO, $Z^{*} Z$ is
 a $C^{*}$-algebra, and so it has a contractive approximate
 identity $(e_{t})$ where $e_{t}$ = $\sum_{k=1}^{n(t)} x^t_{k} y^t_{k}$
 for some $y^t_{k}$ $\in Z$, and $x^t_k = (y^t_k)^*$. It is easy to check
 that $(e_{t})$ is a
cai for $A$, and a similar statement holds for $B$. Indeed it is
clear that $(A, B, X',
Y')$ is a weak$^{*}$-dense strong Morita subcontext of $(M, N, X,
Y)$. Hence $M$ and $N$
are weakly Morita equivalent.  We remark that  
it is proved in \cite{EP} that, in our language,
 $M$ and $N$ are weak* Morita equivalent.

\item Examples of weak and weak* Morita equivalence may also be easily built
as at the end of
  \cite[Section 6]{BJ}, from a weak* closed subalgebra $A$ of a
von Neumann algebra $M$, and a strictly positive $f \in M_+$
satisfying a certain `approximation in modulus' condition. Then the
weak linking algebra of such an example is Morita equivalent in the
same sense to $A$ (see Section 4), but they are probably not always weak* stably isomorphic.  

\item  A beautiful example from \cite{Enew} (formerly part of
\cite{Elef1}): two `similar' separably acting nest algebras 
are clearly weakly Morita equivalent by the facts presented  
around \cite[Theorem 3.5]{Enew} (Davidson's similarity theorem),
but Eleftherakis shows they need not be `$\Delta$-equivalent'
(that is, weak* stably isomorphic \cite{EP}).    \end{enumerate}

In the theory of strong Morita equivalence, and also in our paper, 
it is very important that  $N$  has some kind of `approximate identity' $(f_s)$  of the form 
 \begin{equation} \label{ffa} f_{s}  = \sum_{i=1}^{n_s} \, [y^s_i,x^s_i], \qquad
 \Vert [y^s_1, \cdots , y^s_{n_s}] \Vert \Vert [x^s_1, \cdots , x^s_{n_s}]^T
 \Vert < 1 , \end{equation}
and similarly that $M$ has a cai $(e_t)$
of form \begin{equation} \label{efa} e_t = \sum_{i=1}^{m_t} \,
(x^t_i,y^t_i), \qquad
 \Vert [x^t_1, \cdots , x^t_{m_t}] \Vert
 \Vert [y^t_1, \cdots , y^t_{m_t}]^T
 \Vert < 1 . \end{equation}
Here $x^s_i, x^t_i \in X, y^s_i, y^t_i \in Y$.

In what follows, we say, for example, that
$(\cdot,\cdot)$ is a {\em bimodule map} if $m(x,y) = (mx,y)$ and $(x,y)m = (x,ym)$
for all $x \in X, y \in Y, m \in M$.

\begin{theorem}  \label{neweq}  $(M,N,X,Y)$ is a weak* Morita context
iff the following conditions hold: there exists a separately
weak$^{*}$-continuous completely contractive $M$-bimodule map
$(\cdot,\cdot ) : X \times Y \to M$ which is balanced over $N$, and a
 separately weak$^{*}$-continuous completely contractive  $N$-bimodule
  map  $\lbrack \cdot, \cdot \rbrack : Y \times X \to N$ which is balanced over $M$,
such that $(x,y)x' = x [y,x']$ and $y' (x,y) = [y',x] y$ for $x,x'
\in X, y,y' \in Y$; and also there exist nets $(f_s)$ in $N$ and
$(e_t)$ in $M$ of the form in {\rm (3.1)} and {\rm (3.2)} above,
with $f_s \to 1_N$ and $e_t \to 1_M$ weak*.
\end{theorem}

\begin{proof}
($\Leftarrow$) \ Under these conditions, we first claim that if  $\pi : X \otimes^{\sigma h}_{N} Y$ $\to$ $M$
is the canonical ($w^{*}$-continuous)
$M$-$M$-bimodule map induced by $(\cdot,\cdot)$, then $\pi(u) x \otimes_{N} y $ = $u
(x,y)$ for all $x \in X, y \in Y$, and $u \in X \otimes^{\sigma
h}_{N} Y$.  To see this, fix  $x \otimes_{N} y$ $\in$ $X \otimes^{\sigma h}_{N} Y $. Define
$f,g : X \otimes^{\sigma h}_{N} Y$ $\to$ $X \otimes^{\sigma h}_{N}
Y$ :
 $f(u) = u (x,y)$ and $g(u)  = \pi(u) x \otimes_{N} y $  where $u$ $\in$
 $X \otimes^{\sigma h}_{N} Y$. We need to show that $f$ = $g$. Since
 $X \otimes_{hN} Y$ is $w^{*}$-dense in $X \otimes^{\sigma h}_{N} Y$,
  and $f,g$
 are $w^{*}$-continuous, it is enough to check that
 $f$ = $g$ on $X \otimes_{hN} Y$. For $u$ = $x' \otimes_{N} y'$, we
 have
 $$u(x,y) = x' \otimes_{N} y' (x,y) = x' \otimes_{N} [y',x]y =
 x' [y',x] \otimes_{N} y = (x',y')x \otimes_{N} y = \pi(u)x \otimes_{N}
 y,$$ as desired in the claim.

To see that $M \cong X \otimes^{\sigma h}_{N} Y$,
we shall show that  $\pi$ above is a
 complete isometry.   Since $M$ =  Span${( \cdot, \cdot)}^{-w^{*}}$, it will
 follow from the Krein-Smulian theorem that $\pi$ maps onto $M$.    Choose
 an approximate identity $(e_{t})$ for $A$ of the form in (\ref{efa}).
Define $\rho_{t} : M \to X \otimes^{\sigma h}_{N} Y$ : $\rho_{t} (m)
=
 \sum_{i=1} ^{n_t} m x_{i}^{t} \otimes_{N} y_{i}^{t} $.  For $[u_{jk}] \in
 M_n(X \otimes^{\sigma h}_{N} Y)$, we have by the last paragraph that
 $$\rho_{t}  \circ \pi([u_{jk}]) =
 [ \sum_{i=1} ^{n_t}  \pi(u_{jk}) x_{i} ^{t} \otimes_{N} y_{i}^{t}]
 =  [\sum_{i=1} ^{n_t} u_{jk} (x_{i}^{t}, y_{i}^{t})]
 = [u_{jk} e_{t}]  \buildrel w^* \over \to [u_{jk}] ,$$
the convergence by \cite[Lemma 2.3]{EP}.  Since $\rho_{t}$ is
 completely contractive, we have
$$\lVert [u_{jk}e_{t}] \rVert = 
\lVert (\rho_{t} \circ \pi)([u_{jk}]) \rVert \leq \lVert \pi([u_{jk}]) \rVert .$$
  As $[u_{jk}]$ is the $w^{*}$-limit of the net $([u_{jk}e_{t}])_{t}$, by
 Alaoglu's theorem we deduce that
 $\lVert [u_{jk}] \rVert $ $\leq$ $\lVert \pi([u_{jk}]) \rVert$.
Similarly, $N \cong Y \otimes^{\sigma h}_{M} X$.

($\Rightarrow$)  \  The existence of the nets $(f_s)$ and $(e_t)$
follows from Corollary \ref{ballmod}.  Define $(\cdot,\cdot)$ to be
 the composition of the canonical map $X \times Y \to X \otimes^{\sigma h}_{N} Y$
with the isomorphism of the latter space with $M$.  Similarly one 
obtains $[\cdot,\cdot]$, and these maps have all the desired 
properties except the relations $(x,y)x' = x [y,x']$ and
$y' (x,y) = [y',x] y$.    To obtain these we have to adjust $(\cdot,\cdot)$ by multiplying 
it by a certain unitary in $M$, as in the proof of
\cite[Proposition 1.3]{DB4}.   Indeed that proof transfers easily to our
present setting, and in fact becomes  slightly simpler, since in the latter proof 
the map called $T$ is weak* continuous in our case, and $w^*CB_M(M) \cong M$.   
\end{proof}

\begin{corollary} \label{isws}  Every weak Morita context is
a weak* Morita context.
\end{corollary} 

\begin{proof}  
Let $(M,N,X,Y,( \cdot, \cdot ), \lbrack \cdot,\cdot \rbrack)$
be a weak Morita context with strong Morita subcontext $(A,B,X',Y')$.
 If  $(f_s)$ is a cai  for $B$ it is
clear that $f_s \to 1_N$ weak*. Indeed if a subnet  $f_{s_\alpha}
\to f$ in the weak$^{*}$-topology in $N$, then
  $bf$ = $b$ for all $b \in B$. By  weak$^{*}$-density it
 follows that $bf = b$ for all $b \in N$. Similarly $fb = b$. Thus
 $f = 1_{N}$. 
By Lemma 2.9 in \cite{BMP} we may choose $(f_s)$ of the form
(\ref{ffa}), and similarly $A$ has a cai $(e_t)$
of form in (\ref{efa}).
That $(x,y)x' = x [y,x']$ and $y' (x,y) = [y',x] y$ for
$x,x' \in X, y,y' \in Y$, follows by weak* density, and from
the fact that the analogous relations hold in $X'$ and $Y'$.
Similarly one sees that
$(\cdot,\cdot)$ and $[\cdot,\cdot]$ are balanced bimodule maps.  
\end{proof}   

A key point for us, is that the condition involving (\ref{ffa}) 
in Theorem \ref{neweq}  
becomes a powerful tool when expressed in terms of an `asymptotic factorization' of
$I_Y$ through spaces of the form $C_n(N)$ (or $C_n(B)$ in the case
of a weak Morita equivalence).  Indeed, define $\varphi_s(y)$ to be
the column $[(x_j^s,y)]_j$  in $C_{n_s}(N)$, for $y$ in $Y$, and
define $\psi_s([b_j]) = \sum_j \, y_j^s b_j$ for $[b_j]$ in
$C_{n_s}(N)$. Then $\psi_s(\varphi_s(y)) = f_s y \to y$ weak* if $y
\in Y$ (or in norm if $y \in Y'$, in the case of a weak Morita
equivalence, in which case we can replace $C_{n_s}(N)$ by $C_n(B)$).
Similarly, the condition involving (\ref{efa}) may be expressed
 in terms of an `asymptotic factorization' of $I_X$ through
spaces of the form $R_n(N)$ (or $R_n(B)$ in the `weak Morita' case),
or through $C_n(M)$ (or $C_n(A)$).

Henceforth in this section,  let $(M,N,X,Y,( \cdot,\cdot ),\lbrack
\cdot, \cdot \rbrack)$ be as in Theorem \ref{neweq}. We will also
refer to this $6$-tuple as the weak* Morita context.

\begin{theorem} \label{omod}
Weak* Morita equivalent dual operator algebras have equivalent
 categories of dual operator modules.
\end{theorem}

\begin{proof}
 If $Z \in$ $ _{N} \mathcal{R} $ and if
 $ \mathcal{F}(Z)$ = $X \otimes^{ \sigma h}_{ N} Z$,
 then  $\mathcal{F}(Z)$ is a  left dual operator
 $M$-module by Lemma \ref{D}.  That is,
  $ \mathcal{F}(Z) \in $  $_{M} \mathcal{R}$.
 Further, if $T \in  w^* CB_{N}(Z,W)$, for $Z,W \in$ $ _{N} \mathcal{R}$,
 and if
 $\mathcal{F}(T)$ is defined to be
 $I \otimes_{N} T : \mathcal{F}(Z) \rightarrow \mathcal{F}(W) $,
 then by the functoriality of the normal
 module Haagerup tensor product we have
 $\mathcal{F}(T) \in \, w^{*}CB_{M}(\mathcal{F}(Z), \mathcal{F}(W) )$,  and
 $\lVert \mathcal{F}(T) \rVert_{cb}$
 $\leq$ $ \lVert T \rVert_{cb}$. Thus $\mathcal{F}$ is a contractive
 functor from $_{N} \mathcal{R}$ to $_{M} \mathcal{R}$. Similarly,
  we obtain
 a contractive functor $\mathcal{G}$ from $_{M} \mathcal{R}$ to
 $_{N} \mathcal{R}$.  Namely, $\mathcal{G}(W)$ = $Y \otimes^{\sigma h}_{M} W$,
 for $W \in $ $_{M} \mathcal{R}$, and $\mathcal{G}(T)$ =
 $I \otimes_{M} T$ for $T \in w^*CB_{M}(W,Z)$ with
 $W, Z $ $\in$ $ _{M} \mathcal{R}$.
 Similarly, it is easy to check that
  these functors are completely contractive; for
 example, $T \mapsto \mathcal{F}(T)$ is a completely contractive map on
 each space $w^*CB_{N}(Z,W)$ of morphisms. If we compose $\mathcal{F}$
 and $\mathcal{G}$, we find that for $Z \in$ $_{N} \mathcal{R}$ we have
 $\mathcal{G}(\mathcal{F}(Z)) \in $ $_{N} \mathcal{R}$.  By  Proposition \ref{asso2}
and Lemma \ref{C}, we have
$$\mathcal{G}(\mathcal{F}(Z))  \cong Y \otimes^{\sigma h}_{M}
(X \otimes^{\sigma h}_{N} Z)
 \cong  (Y \otimes^{\sigma h}_{M} X) \otimes^{\sigma h}_{N} Z
\cong  N \otimes^{\sigma h}_{N} Z \cong  Z.  $$ \noindent where the
isomorphisms are completely isometric. The rest of
 the proof follows as in Theorem 3.9 in \cite{BMP}.
\end{proof}

{\bf Remark.}  We imagine that the ideas of \cite{DB4} show that the converse 
of the last theorem is true, and hope to pursue this in the future.

\medskip
 
We shall adopt the convention
 from algebra of writing maps on the side opposite the one on which ring
 acts on the module. For example a left $A$-module map will be written
 on the right and a right $A$-module map will be written on the left.
  The pairings and actions arising in the weak Morita context
 give rise to eight maps:

\begin{align*}
  R_{N}:  N  &\to  CB_{M}(X,X),      & xR_{N}(b) &=  x \cdot b \\
  L_{N}:  N  &\to  CB(Y,Y)_{M},      & L_{N}(b)y &=   b \cdot y \\
  R_{M}:  M  &\to  CB_{N}(Y,Y),      & yR_{M}(a) &=   y \cdot a \\
  L_{M}:  M  &\to  CB(X,X)_{N},      & L_{M}(a)x &=   a \cdot x \\
 R^{M}:  Y &\to  CB_{M}(X,M),      & xR^{M}(y) &=   ( x,y ) \\
 L^{N}:  Y &\to  CB(X,N)_{N},      & L^{N}(y)x &=   [y,x] \\
 R^{N}:  X &\to  CB_{N}(Y,N),      & yR^{M}(x) &=   [y,x] \\
 L^{M}:  X &\to  CB(Y,M)_{M},      & L^{M}(x)y &=   ( x,y )
  \notag
\end{align*}

The first four maps are completely contractive since  module actions
 are completely contractive. Also the maps $L_{N}$ and $L_{M}$ are
 homomorphisms and $R_{N}$ and $R_{M}$ are anti-homomorphisms.
 Similar proofs to the analogous results in \cite{BMP}
 show that $R^M, L^{N}$, $ R^{N}$, and $ L^{M}$ are completely contractive.

\begin{theorem} \label{three}
 If $(M, N , X, Y, ( \cdot, \cdot ), \lbrack \cdot, \cdot \rbrack )$ is
 a weak* Morita context, then each of the maps $R^{M}$, $R^{N}$, $L^{M}$
 and $L^{N}$ is a weak* continuous complete isometry. The range of
  $R^{M}$ is $w^{*}$$ CB_{M}(X,M)$, with similar assertions holding
   for $R^{N}$, $L^{M}$ and
  $L^{N}$. The map $L_{N}$ (resp.  $R_{N}$) is a $w^*$-continuous
  completely isometric
 isomorphism (resp.  anti-isomorphism) onto the
 $w^{*}$-closed left (resp. right) ideal  $w^*CB(Y)_M$
(resp.\ $w^*CB_M(X)$).  The latter also equals the left multiplier
algebra (see \cite[Chapter 4]{BLM}) ${\mathcal M}_{\ell}(Y)$ (resp.\
${\mathcal M}_r(X)$). Similar results hold for
 $L_{M}$ and $R_{M}$.
\end{theorem}

\begin{proof}   Most of this can be proved directly, as in
\cite[Theorem 4.1]{BMP}.  Instead we will deduce it from the
functoriality (Theorem \ref{omod}).  For example, because of the equivalence of
categories
via the functor ${\mathcal F} = Y \otimes^{\sigma h}_M
-$, we have completely isometrically:
$$M \cong \, w^*CB_M(M) \cong \, w^*CB_N({\mathcal
F}(M)) \cong \, w^*CB_N(Y) ,$$ and the composition of these maps is
easily seen to be $R_M$. Thus $R_M$ is a  complete isometry. Similar
proofs work for the other seven
 maps.  To see that $L_{N}$ is $w^{*}$-continuous, for example, let $(b_{t})$
be a bounded net in
 $N$ converging in the $w^{*}$-topology of $N$ to $b$ $\in N$. Then
 $L_{N}(b_{t})$ is a bounded net in $CB(Y)_{M}$. As the module action is
 separately $w^{*}$-continuous, it is easy to see that
 $L_{N}(b_{t})$ converges to $ L_{N}(b)$ in the $w^{*}$-topology.  Thus
 $L_N$ is a $w^{*}$-continuous isometry with  $w^{*}$-closed range,
  by the Krein-Smulian theorem.
 To see that its range is a left ideal simply
 use the weak$^{*}$-density of the span of terms $[y,x]$ in $N$, and the
 equation $T L_{N}([y,x])(y')$ = $ L_{N}[Ty, x](y')$ for  $T$
 $\in CB(Y,Y)_{M}$, $ y' \in Y$.  We leave the variants for the other maps
 to the reader.

 To see the
assertions involving multiplier algebras, note that we have obvious
completely contractive maps
$$N \longrightarrow {\mathcal M}_{\ell}(Y)  \longrightarrow
w^{*} CB(Y)_{M} .$$ The first of these arrows arises since $Y$ is a
left operator $N$-module (see \cite[Theorem 4.6.2]{BLM}).  The
second arrow always exists by general properties (see e.g.\
\cite[Chapter 4]{BLM}, or Theorem 4.1 in \cite{BM2}) of the left
multiplier algebra of a dual operator module.
  Both arrows are weak* continuous by
e.g.\ Theorem 4.7.4 (ii) and 1.6.1 in \cite{BLM}. Since $N  \cong
w^{*} CB(Y)_{M}$ completely isometrically and
$w^*$-homeomorphically, we deduce that these spaces coincide with
${\mathcal M}_{\ell}(Y)$ too.
 \end{proof}

{\bf Remark.}   Note that in the case of weak Morita equivalence,
 $CB_{A}(X')$ is an operator algebra (\cite{BMP}, Theorem
 4.9). It is not true in general that $CB_{M}(X)$ is an operator
 algebra,
as we show in \cite{UK}.  Nonetheless,  the above shows that $w^{*}
CB_M(X)$ is a dual operator algebra ($\cong N$).

\begin{theorem}
If $M$ and $N$ are weak* Morita equivalent dual operator algebras,
then their centers are completely isometrically isomorphic via a
$w^*$-homeomorphism.
\end{theorem}

\begin{proof}
 By Theorem \ref{three} there is a $w^*$-continuous complete isometry
 $R_M : M \to w^*CB_N(Y)$.  The restriction of $R_M$ to
 $Z(M)$ maps into $w^*CB(Y)_M \cong N$, and so we have defined
 a $w^*$-continuous completely isometric homomorphism
 $\theta : Z(M) \to N$.  One easily
 sees that $\theta(a)(y) = y a$, for $a \in Z(M)$. It is also
 easy to see that this implies that $\theta$ maps into $Z(N)$,
  and to argue, by symmetry, that $\theta$ must be an isomorphism.
\end{proof}

\begin{lemma} \label{ball}  If $Z$ is a left dual operator $N$-module,
then the canonical map from $X \otimes_{hN} Z$ into $X
\otimes^{\sigma h}_{N} Z$ is completely isometric. In the case of
weak Morita equivalence, the canonical map $X' \otimes_{hB} Z \to X
\otimes^{\sigma h}_{N} Z$ is completely isometric, and it maps the
ball onto a $w^{*}$-dense set in {\rm Ball}$(X \otimes^{\sigma
h}_{N} Z)$.
\end{lemma}

\begin{proof}  We just treat the `weak Morita' case, the other being
similar.  The canonical map here is completely contractive, let us
call it $\theta$. On the other hand, let $\varphi_s, \psi_s$ be as
defined just below Corollary  \ref{isws}, with $\psi_s(\varphi_s(y)) = f_s y \to y$.  Then
for $u \in M_n(X' \otimes_B Z)$, we have
$$\Vert \theta_n(u) \Vert \geq \Vert (\varphi_s \otimes
I)_n(\theta_n(u)) \Vert = \Vert (\varphi_s \otimes I)_n(u) \Vert
\geq \Vert f_s u \Vert .$$ Taking a limit over $s$, gives $\Vert
\theta_n(u) \Vert \geq \Vert u \Vert$.

Let $u$ $\in {\rm Ball}(X \otimes^{\sigma h}_{N} Z)$.  By Lemma
\ref{ballmod}, there exists a net $(u_{t})$ in the image of {\rm
Ball}$(X \otimes_{h N} Z)$ such that $u_{t}$ $ \buildrel w^* \over
\to$ $u$.
 We may assume that each $u_{t}$ is of the form
 $w \odot z$, for $w \in {\rm Ball}(R_n(X)), z \in {\rm Ball}(C_n(Z))$.
Let $(e_t)$ be as in (\ref{efa}), that is, with each $e_t$ of the
form $(x,y)$ (in
 suggestive notation), for $x \in
{\rm Ball}(R_m(X'))$ and $y \in {\rm Ball}(C_m(Y'))$.    However, $w
\odot z$ is the weak* limit of terms $e_t w \odot z$, and $e_t w
\odot z = x \odot v$, where $v$ is a column with $k$th entry $\sum_j
[y_k, w_j] z_j$.  It is easy to check that $\Vert v \Vert \leq 1$.
\end{proof}

\begin{proposition} \label{equiv}
Weak* Morita equivalence is an equivalence relation.
\end{proposition}

\begin{proof}  This follows the usual lines, for example the transitivity
follows from associativity of the tensor products and Lemma \ref{C}.
\end{proof}

{\bf Remark.}  Concerning transitivity of weak Morita equivalence,
it is convenient to consider Definition \ref{basicdef} as defining
an equivalence between pairs $(M,A)$ and $(N,B)$, as opposed to just
between $M$ and $N$. That is we also consider the
 weak$^{*}$-dense operator subalgebras.  Then it is fairly routine to
 see that weak Morita equivalence
 is an equivalence relation \cite{UK}.

\begin{theorem} \label{hmod}
Weak* Morita equivalent dual operator algebras have equivalent
 categories of normal Hilbert modules.  Moreover, the equivalence
preserves the subcategory of modules corresponding to completely
isometric normal representations.
\end{theorem}

\begin{proof}  If $H$ is a normal Hilbert $M$-module, 
let $K = Y \otimes^{\sigma h}_{M} H^{c}$.
By the discussion just below Corollary  \ref{isws},
combined with Corollary \ref{F}, there are nets of maps $\varphi_s :
K \to C_{n_s}(M) \otimes^{\sigma h}_M \, H^c \cong C_{n_s}(H^c)$,
and maps $\psi_s : C_{n_s}(H^c) \to K$, with $\psi_s (\varphi_s(z))
= f_s z \to z$ weak* for all $z \in K$. Here $(f_s)$ is as in (3.1).
Let $\Lambda$ be the directed set indexing $s$, and let ${\mathcal
U}$ be an ultrafilter with the property that $\lim_{{\mathcal U}} \,
z_s = \lim_{\Lambda} \, z_s$ for scalars $z_s$, whenever the latter
limit exists.  Let $H_{\mathcal U}$ be the ultraproduct of the
spaces $C_{n_s}(H^c)$, which is a column Hilbert space, as is well
known and easy to see.  Define $T : K \to H_{\mathcal U}$ by $T(x) =
(\varphi_s(x))_s$, for $x \in K$.  This is a complete
contraction. To see that it is an isometry, note that for any $x \in
K, \rho \in {\rm Ball}(X_*)$, we have
$$|\rho(x)| = \lim_{{\mathcal U}} \, |\rho(\psi_s(\varphi_s(x)))| \leq
\lim_{{\mathcal U}} \, \Vert \varphi_s(x) \Vert =  \Vert T(x) \Vert
.$$ Similarly, $T$ is a complete isometry, as we leave to the reader
to check.

By Lemma \ref{ball} and Corollary \ref{ballmod}, $Y  \otimes_{hM}
H^c = Y  \otimes^{\sigma h}_{M} H^{c}$.  Note that since $-
\otimes_{h} H^{c} = - \overset{\frown}{\otimes} H^c$ (see e.g.\
\cite[Proposition 9.3.2]{ER1}), we may replace $\otimes_{hM}$ here
by $\overset{\frown}{\otimes}_M$ (see 3.4.2 of \cite{BLM} for this
notation).

That $K = Y \otimes^{\sigma h}_{M} H^{c}$ is a normal Hilbert
$N$-module follows from Theorem \ref{omod}. Finally, suppose that
$M$ is a weak* closed subalgebra of $B(H)$, we will show that the
induced representation $\rho$ of $N$ on $K$ is completely isometric.
 Certainly this map is completely contractive. If $[b_{pq}] \in M_{d}(N)$,
$[y_{kl}] \in {\rm Ball}(M_m(Y)), [\zeta_{rs}] \in {\rm
Ball}(M_g(H^c)), [x_{ij}] \in {\rm Ball}(M_n(X))$, then
$$\Vert [\rho(b_{pq})] \Vert \geq \Vert [b_{pq}y_{kl} \otimes \zeta_{rs}] \Vert
\geq \Vert [ (x_{ij},b_{pq} y_{kl}) \zeta_{rs} ] \Vert .$$  Taking
the supremum over all such $[\zeta_{rs}]$, gives
$$\Vert [\rho(b_{pq})] \Vert \geq \sup \{ \Vert [(x_{ij},b_{pq} y_{kl})] \Vert :
[x_{ij}] \in {\rm Ball}(M_n(X)) \} = \Vert [b_{pq} y_{kl}] \Vert,
$$
by Theorem \ref{three}.  Taking the supremum over all such
$[y_{kl}] \in {\rm Ball}(M_m(Y))$
gives $\Vert [\rho(b_{pq})] \Vert \geq \Vert [b_{pq}] \Vert$,
by Theorem \ref{three}
again.   \end{proof}

The last result shows that weak* Morita equivalent
operator algebras have equivalent categories of normal
representations.  It would be very interesting to characterize when two
operator algebras have equivalent categories of normal
representations; it seems quite possible that this happens iff 
they are weak* Morita equivalent.  

\begin{corollary} \label{mazur}  If $H \in \, _M{\mathcal H}$ then
 $Y \otimes^{\sigma h}_{M} H^{c} =
 Y \otimes_{h M} H^{c} = Y \overset{\frown}{\otimes}_{M} H^{c}$
 completely isometrically.
These are column Hilbert spaces.  Here $\overset{\frown}{\otimes}_{M}$ 
is as in {\rm 3.4.2} of {\rm \cite{BLM}}. In the case of weak Morita
equivalence, these also equal $Y' \otimes_{h A} H^{c} = Y'
 \overset{\frown}{\otimes}_{A} H^{c}$.
 \end{corollary}

\begin{proof}  We saw the first part in the last proof.  The assertions
involving $Y'$ follow in a similar way, by Lemma \ref{ball}.  Note
that in this case, if $\eta \in H \ominus [A H]$ then $$\langle \eta
, \eta \rangle =
 \lim_t \, \langle e_t \eta , \eta \rangle = 0 .$$
Thus $A$ acts nondegenerately on $H$.
\end{proof}

\section{The weak linking algebra}

In this section again, $(M,N,X,Y,( \cdot, \cdot ), \lbrack \cdot,\cdot
\rbrack)$ is a weak*
 Morita context.
Suppose that $M$ is represented as a weak$^{*}$-closed nondegenerate
subalgebra of $B(H)$, for a Hilbert space $H$.
 Then by Corollary \ref{mazur}, $K= Y
\otimes^{\sigma h}_{M} H^{c}$ is a column Hilbert space. Define a
right $M$-module map $\Phi : Y \to B(H,K)$ by $\Phi(y)(\zeta) = y
\otimes_{M} \zeta$ where $y \in Y$ and $\zeta \in H$. It is easy to
see that $\Phi$ is a completely contractive $N$-$M$-bimodule map. It
is weak$^{*}$-continuous, since if we have a bounded net $y_t \to y$
 weak$^{*}$ in $Y$, and if $\zeta \in H$, then
$y_t \otimes_{M} \zeta \to y \otimes_{M} \zeta$ weakly by \cite{EP}.
That is, $\Phi(y_t) \to \Phi(y)$ in the WOT, and it follows that
$\Phi$ is weak$^{*}$-continuous.  If $\Vert \Phi(y) \Vert \leq 1$,
and if $\zeta \in {\rm Ball}(H^{(n)})$, and $[x_{ij}] \in {\rm Ball}(M_n(X))$,
then
$$\Vert [(x_{ij},y)] \zeta \Vert  =  \Vert [x_{ij} \otimes \Phi(y)]
\zeta \Vert \Vert \leq \Vert \Phi(y) \Vert.$$ Taking the supremum
over such $\zeta$, and then over such $[x_{ij}]$, we obtain from
Theorem \ref{three} that $\Vert y \Vert \leq 1$.  Thus $\Phi$ is an
isometry, and a similar but more tedious argument shows that $\Phi$
is a complete isometry.  By the Krein-Smulian theorem we deduce that
the range of $\Phi$ is weak$^{*}$-closed.  A similar
argument, which we leave to the reader, shows that the map $\Psi : X
\to B(K,H)$, defined by $\Psi(x)(y \otimes \zeta) = (x,y) \zeta$, is
a $w^*$-continuous completely isometric $M$-$N$-bimodule map. As we
said in Theorem \ref{hmod}, the induced normal representation $N \to
B(K)$ is completely isometric.

We use the above to define the direct sum $M \oplus^{c} Y$ as
follows.   For specificity, the reader might want to
 take $H$ to be a universal normal
representation of $M$, that is the restriction to $M$ of a
one-to-one normal representation of $W^*_{\rm max}(M)$. Define a map
$\theta : M \oplus^{c} Y \to B(H, K \oplus H)$ by $\theta
((m,y))(\zeta) = (m \zeta , y \otimes_{M} \zeta)$,
 for $y \in Y, m \in M, \zeta \in H$.
One can quickly check that $\theta$ is a one-to-one, $M$-module map,
and that $\theta$ is a weak$^{*}$-continuous  complete isometry when
restricted to each of $Y$ and $M$.  Also, $W$= Ran$(\theta)$ is
easily seen to be weak$^{*}$-closed. We norm $M \oplus^{c} Y$  by
pulling back the operator space structure
 from $W$ via $\theta$.
Thus $M \oplus^{c} Y$ may be identified with the weak$^{*}$-closed
right $M$-submodule $W$ of $B(H, H \oplus K)$; and hence it is a
dual operator $M$-module.  In a similar way, we define $M \oplus^r
X$ to be the canonical weak$^{*}$-closed left $M$-submodule of $B(H \oplus K,H)$.

We next define the `weak linking algebra' of the context, namely
  $$\mathcal{L}^w =
\Big\lbrace \left[
\begin{array}{ccl}
a & x \\
y & b
\end{array}
\right] : \ a \in M, b\in N, x\in X, y\in Y \Big\rbrace,  $$ with
the obvious multiplication.   As in \cite[Lemma 5.6]{BMP}, one
easily sees that there is at most one possible sensible dual
operator space structure on this linking algebra.
 Indeed if $\Lambda$ is the set indexing $t$ in
the net in (\ref{efa}), and if $\beta, t \in \Gamma$, then define
$\theta^{\beta,t}$ on the linear space $\mathcal{L}^w$ to be the map
$\theta^{\beta}$ in  \cite[p.\ 45]{BMP}, but with all the
$y_i^\beta$ replaced by $y_i^t$.  Then a simple modification of the
argument in  \cite[p.\ 50-51]{BMP}, and using semicontinuity of the
norm in the weak* topology, yields that
any `sensible' norm assigned to
$\mathcal{L}^w$ must agree with $\sup_{\beta, t} \, \Vert
\theta^{\beta,t}(\cdot) \Vert$

That such a dual operator space structure does exist, one only need
 view $\mathcal{L}^{w}$ as a subalgebra
${\mathcal R}$ of $B(H \oplus K)$, using the obvious pairings $X
\times K \to H$ (induced by $(\cdot,\cdot)$), $Y \times H \to K$,
and $N \times K \to K$ (this is the induced representation of $N$ on
$K$ from Theorem \ref{hmod}). It is easy to check that $(M,{\mathcal
R},M \oplus^r X,M \oplus^{c} Y)$ is also a  weak* Morita context
(this follows from norm equalities of the kind in e.g.\ 
the centered equations
in \cite[Theorem 5.12]{BMP}). This all may be most easily visualized by
picturing both contexts as $3 \times 3$-matrices, namely as
subalgebras of $B(H \oplus H \oplus K)$. Theorem \ref{three} gives
${\mathcal R} \cong w^{*} CB(M \oplus^{c} Y)_{M}$ completely
isometrically and $w^*$-homeomorphically.

 Note that in a weak Morita situation, the linking operator
algebra of the strong Morita context $(A,B,X',Y')$ can be identified
completely isometrically as the obvious weak* dense subalgebra
${\mathcal L}$ of ${\mathcal R}$ (see e.g.\ \cite[Proposition
6.10]{DB2}). Incidentally, at this point we have already proved the
assertion made at the start of Example (1) in Section 3, and indeed
that every weak Morita equivalence arises as the weak* closure of a
strong Morita equivalence, or can be viewed as the weak*-closure, in
some representation, of the linking operator algebra of a strong
Morita equivalence. We
 have a strong Morita context $(A,{\mathcal L},A
\oplus^r X',A \oplus^c Y')$ (see \cite{BMP,BMN}), which can be
viewed as a subcontext of $(M,{\mathcal R},M \oplus^r X,M \oplus^{c}
Y)$.  Thus the latter  is a weak Morita context.

Extracting from the last paragraphs, we have:

 \begin{corollary} \label{ela}  $M$ is weak* Morita equivalent to
the weak linking algebra ${\mathcal L}^w$.  Indeed this is a weak
Morita equivalence if $(M,N,X,Y)$ is a weak Morita context.
\end{corollary}

It is often useful here to know that:

\begin{proposition} \label{assoc}  With notation as in Theorem
{\rm \ref{hmod}}, we have $(M \oplus^c Y) \otimes^{\sigma h}_M H^c
\cong (H \oplus K)^c$ as Hilbert spaces.
\end{proposition}

\begin{proof}  We will just sketch this, since it is
not used here.  By Corollary \ref{ela}, and Theorem {\rm
\ref{hmod}}, we have that $L = (M \oplus^c Y) \otimes^{\sigma h}_M
H^c$ is a column Hilbert space. Moreover, the projections from $M
\oplus^c Y$ onto $M$ and $Y$ respectively, induce by Corollary
\ref{F}, projections $P$ and $Q$ from $L$ onto $M \otimes^{\sigma
h}_M H^c \cong H^c$, and $K$, respectively, such that $P + Q = I$.
\end{proof}

Mimicking the proof of \cite[Theorem 5.1]{BMP} we have:

\begin{theorem}
Let  $(M,N,X,Y)$ be a weak* Morita context. Then there is a lattice
isomorphism between the $w^{*}$-closed $M$-submodules of $X$ and the
lattice of $w^{*}$-closed left ideals in $N$. The $w^{*}$-closed
$M$-$N$-submodules of $X$ corresponds to the $w^{*}$-closed
two-sided ideals in $N$. Similar statements for $Y$ follows by
symmetry. In particular, $M$ and $N$ have isomorphic lattices of
$w^{*}$-closed two-sided ideals.
\end{theorem}

We next show, analogously to
 \cite[Section 6]{BMP}, that if $M$ and $N$ are
 $W^*$-algebras, then they are Morita equivalent in Rieffel's sense iff
 they are weakly (or equivalently, weak*) Morita equivalent in our sense.  Indeed we already
 have remarked (Example (2) in Section 3) that Rieffel's
Morita equivalence is an example of our weak Morita equivalence. The
following gives the converse, and more:

\begin{theorem} \label{alsodin}
Let $(M,N,X,Y)$  be a
 weak* Morita context where $N$ is a  $W^*$-algebra. Then $M$ is
 a $W^*$-algebra, and there is a
 completely isometric isomorphism $i : \overline{X} \rightarrow Y$ such
 that $X$ becomes a  $W^{*}$-equivalence $M$-$N$-bimodule (see e.g.\
{\rm 8.5.12} in \cite{BLM} with inner
 products defined by the formulas
$_M\langle x_{1} ,x_{2} \rangle $ = $(x_{1},i(\bar{x_{2}}))$ and
$\langle  x_{1} , x_{2}\rangle_N $ = $\lbrack i(\bar{x_{1}}),x_{2}
\rbrack$.
\end{theorem}

\begin{proof}  First we represent the linking algebra on a Hilbert
space $H \oplus K$ as above.  We rechoose the net $(e_t)$ such that
$e_t \to I_H$ strongly, so that $e_t^* e_t \to I_H$
thus weak*, and similarly for the net $(f_s)$.   To accomplish this,
note that the WOT-closure of the convex hull of the $(e_t)$ equals
the SOT-closure, by elementary operator theory.  However it is easy
to see that the form in (\ref{ffa}) is preserved if we
replace  $e_s$ by convex combinations of the $e_t$.  Now one can
follow the proof of \cite[Theorem 6.2]{BMP} to deduce that the
adjoint of any $y \in Y$ is a limit of terms in $X$. That is $Y
\subset X^*$.  Similarly, $X \subset Y^*$. So $X = Y^*$, and so it
follows that $M$ is a $W^*$-algebra, and $X$ is a WTRO (this term
was defined in the list of examples in Section 3) setting up a
$W^*$-algebra Morita equivalence.  We leave the rest as an exercise.
\end{proof}

The following is the nonselfadjoint analogue of a theorem of
Rieffel.   A special case of it is mentioned, with a
proof sketch, at the end of \cite{BM2}.

\begin{theorem} \label{didn} Let $H$ be a universal normal representation for $M$,
and let $K$ be the induced representation of $N$ studied above. Then
$M' \cong N'$; that is there is a completely isometric
$w^*$-continuous isomorphism $\theta : B_M(H) \cong B_N(K)$. Writing
${\mathcal R}$ for either of these commutants, we have $X \cong \,
B_{\mathcal R}(K,H)$ and $Y  \cong \, B_{\mathcal R}(H,K)$
completely isometrically and as dual operator bimodules.
\end{theorem}

\begin{proof}
One uses the equivalence of categories to see that $B_M(H) \cong
B_N(\mathcal{F}(H)) = B_N(K)$ completely isometrically, in the
notation of Theorem \ref{omod}.   That is, $M' \cong N'$ as
asserted, and it is easy to argue that if $\theta$ is this
isomorphism then $\Phi(y) T = \theta(T) \Phi(y)$ for all $y \in Y, T
\in M'$.  Here $\Phi$ is as in the discussion at the start of
Section 4.  Now mimic the proof of 8.5.32 and 8.5.37 in \cite{BLM}.
The main point to bear in mind is that since $M$ is weak* Morita
equivalent to the weak linking algebra ${\mathcal L}^w$, the induced
representation of ${\mathcal L}^w$ is also a universal normal
representation, by easy category theoretic arguments.   Thus by
\cite{BSo} it satisfies the double commutant theorem.  Carefully
computing the first, and then the second, commutants of ${\mathcal
L}^w$ as in 8.5.32 in \cite{BLM}, and using the double commutant
theorem, gives the result.
\end{proof}

\example  If $M$ and $N$ are {\em finite dimensional} then weak*
Morita equivalence equals strong Morita equivalence, and coincides
also with the equivalence considered in \cite{Elef1,Elef2}, that is,
weak* stable isomorphism \cite{EP}. Indeed if $(M, N, X, Y)$ is a
weak* Morita context, then it is clearly a strong Morita context,
and by \cite[Lemma 2.8]{BMP} we can actually factor the identity map
$I_Y$ through $C_n(M)$ for some $n \in \Ndb$, so that $Y$ is finite
dimensional. Similarly, $X$ is finite dimensional.   To see that
this implies that $M$ and $N$ are weak* stably isomorphic, note that
in this situation, since
 $M \cong X \otimes^{\sigma h}_{N} Y $, there is
 a norm $1$ element in $X \otimes_h Y$ mapping to $1_M$.
 Similarly for $1_N$, and it is evident that
 one has  what is called a `quasi-unit of norm $1$' in
\cite[Section 7]{BMP}.  By \cite[Corollary 7.9]{BMP}, $M$ and $N$
are stably isomorphic, and taking second duals and using e.g.\
(1.62) in \cite{BLM}, we see that they are weak* stably isomorphic.
In the infinite dimensional case however, all these notions are distinct.

\section{Morita equivalence of generated $W^{*}$-algebras}

From \cite{BMN} or \cite{DB2}, we know that a strong Morita
equivalence of operator algebras in the
 sense of \cite{BMP} `dilates' to, or is a subcontext of,
a strong
 Morita equivalence in the sense of Rieffel, of  containing
 $C^*$-algebras.  This happens in a very tidy way.  More
 particularly, suppose that $ (A, B , X, Y )$
is a strong Morita context of operator algebras $A$ and $B$.
 Then  any $C^*$-algebra $C$ generated by $A$ induces
 a $C^*$-algebra $D$ generated by $B$, and $C$ and $D$ are
  strongly
 Morita equivalent in the sense of Rieffel \cite{RF1}, with
 equivalence bimodule the `$C^*$-dilation' (see \cite{DB6})
 $C \otimes_{hA} X$.  Moreover the linking algebra
 for $A$ and $B$ is (completely isometrically)
a subalgebra of the linking $C^*$-algebra
 for $C$ and $D$.   We see next that all of this, and the accompanying
 theory, will extend to
 our present setting.  Although one may use any `$W^*$-cover' in
the arguments below, for specificity, the maximal $W^{*}$-algebra $W^*_{\rm
 max}(M)$ from \cite{BSo} will take the place of $C$ above, and
the `maximal $W^*$-dilation'
 $W^*_{\rm max}(M) \otimes^{\sigma h}_M X$ will play the role of
 the $C^*$-dilation.   One can develop a theory for this
`$W^*$-dilation' in a general setting analogously to \cite{DB6,BMN},
but we shall not take the time to do this here (see \cite{UK}). We
will however state that just as in \cite{DB6}, any (left, say)
dual operator
$M$-module
is completely isometrically embedded in its `maximal
$W^*$-dilation', via the $M$-module map $x \mapsto 1 \otimes x$, which is weak*
continuous.

Throughout this section again, $(M,N,X,Y)$ is a
 weak* Morita context.
We shall show that the `left' and `right'
 $W^*$-dilations coincide, and constitutes a bimodule implementing
 the $W^*$-algebraic Morita  equivalence between $W^*_{\rm
max}(M)$ and $W^*_{\rm max}(N)$.

\begin{theorem} \label{wiscs}
The $W^*$-dilation $Y \otimes^{\sigma h}_M \, W^*_{\rm max}(M)$ is a
right $C^*$-module over $W^*_{\rm max}(M)$.
\end{theorem}

\begin{proof}   With $H$ a normal universal Hilbert $M$-module as usual, we
may view $W^*_{\rm max}(M)$ as the von Neumann algebra ${\mathcal
R}$ generated by $M$ in $B(H)$.  Let $K = Y \otimes^{\sigma h}_M \,
H^c$ as usual, and let $Z = Y \otimes^{\sigma h}_M \, W^*_{\rm
max}(M)$.  Note that
$$Z \otimes^{\sigma h}_{W^*_{\rm max}(M)} H^c \cong Y
\otimes^{\sigma h}_M \, W^*_{\rm max}(M) \otimes^{\sigma
h}_{W^*_{\rm max}(M)} H^c \cong Y \otimes^{\sigma h}_M \, H^c = K
.$$ This allows us to define a completely contractive
weak$^{*}$-continuous $\phi : Z \to B(H,K)$ given by $\phi(y \otimes
a)(\zeta) = y \otimes a \zeta$, for $y \in Y, a \in {\mathcal R},
\zeta \in H$. Note that $\phi$ restricted to the copy of $Y$ is just
the map $\Phi$ at the start of Section 4.
  We are following the ideas of \cite[p.\
286-288]{DB1}.  It is clear that $\phi$ is a ${\mathcal R}$-module
map.   By the discussion just below Corollary  \ref{isws}, combined with Corollary \ref{F},
there are nets of maps $\varphi_s \otimes I : Z \to C_{n_s}(M)
\otimes^{\sigma h}_M \, W^*_{\rm max}(M) \cong C_{n_s}(W^*_{\rm
max}(M))$, and maps $\psi_s \otimes I$, with $(\psi_s  \otimes
I)(\varphi_s \otimes I) (z) = f_s z \to z$ weak* for all $z \in Z$.
Here $(f_s)$ is as in (\ref{ffa}), and the last convergence follows
from e.g.\ \cite[Lemma 2.3]{EP}. We have $\Vert [ f_s z_{ij} ] \Vert
\leq \Vert [(\varphi_s \otimes I)(z_{ij})] \Vert \leq \Vert
[\phi(z_{ij})] \Vert$.  This follows, as in \cite[p.\ 287]{DB1},
from the fact that there is a sequence of weak* continuous complete
contractions
$$B(H,K) \to B(H,C_{n_t}(M) \otimes^{\sigma h}_M \,
W^*_{\rm max}(M) \otimes^{\sigma h}_{W^*_{\rm max}(M)} H^c) \cong
B(H,C_{n_t}(H^c))$$ that maps $\phi(y \otimes a)$ to $\varphi_s(y)
a$, for $y \in Y, a \in {\mathcal R}$, and hence maps $\phi(z)$ for
 $z \in Z$, to $(\varphi_s \otimes I)(z)$.  As in \cite[p.\ 287]{DB1},
 it follows that $\phi$ is a complete
isometry.

Define $\langle z , w \rangle = \phi(z)^* \phi(w)$ for $z, w \in Z$.
To see that this is a ${\mathcal R}$-valued inner product on $Z$, we
will use von Neumann's double commutant theorem (this is a well
known idea). Note that if $\Delta(A) = A \cap A^*$ is the `diagonal'
of a subalgebra of $B(H)$, then ${\mathcal R}' = \Delta(M')$, the
`prime' denoting commutants.   The proof of Theorem \ref{didn} shows
that there is a completely isometric isomorphism $\theta : M' \to
N'$, such that $\Phi(y)T = \theta(T) \Phi(y)$ for $y \in Y, T \in
M'$, where $\Phi(y)(\zeta) = y \otimes \zeta \in K$, for $\zeta \in
H$. By 2.1.2 in \cite{BLM}, $\theta$ restricts to a $*$-isomorphism
from $\Delta(M') = {\mathcal R}'$ onto $\Delta(N')$. It follows
that, in the notation of Theorem 5.1, if $y \in Y, a \in {\mathcal
R}, \zeta \in H, T \in M'$ that
$$\phi(y \otimes a)(T \zeta) = y \otimes a T \zeta = y \otimes T a
\zeta = \Phi(y) T (a \zeta) = \theta(T)  \Phi(y) (a \zeta) =
\theta(T) \phi(y \otimes a)(\zeta) .$$  Hence if $w, z \in
Z$ then
$$\phi(z)^* \phi(w) T = \phi(z)^* \theta(T) \phi(w) =
(\theta(T^*) \phi(z))^* \phi(w) = (\phi(z) T^*)^* \phi(w) = T
\phi(z)^* \phi(w) ,$$ so that $\phi(z)^* \phi(w) \in {\mathcal R}''
= {\mathcal R}$.

Thus $Z$ is a right $C^*$-module over $W^*_{\rm max}(M)$, completely
isometrically isomorphic to the WTRO Ran$(\phi)$.
\end{proof}

\begin{theorem} \label{wismo}
Suppose that $(M,N,X,Y)$ is a weak* Morita context.  Then
$W^{*}_{\rm max}(M)$ and $W^{*}_{\rm max}(N)$ are Morita equivalent
$W^{*}$-algebras in the sense of Rieffel, and the associated
equivalence bimodule is $Y \otimes^{\sigma h}_M \, W^*_{\rm
max}(M)$.  Moreover, $Y \otimes^{\sigma h}_M \, W^*_{\rm max}(M)
\cong W^*_{\rm max}(N) \otimes^{\sigma h}_N Y$ completely
isometrically.  Analogous assertions hold with $Y$ replaced by $X$.
 Finally, the $W^*$-algebra linking algebra for this Morita
equivalence contains completely isometrically as a subalgebra the
linking algebra ${\mathcal L}^w$ defined earlier for the context
$(M,N,X,Y)$.
\end{theorem}

\begin{proof}  We use the idea in \cite[p.\ 406-407]{DB2}
and \cite[p.\ 585-586]{BMN}. Let $H, K$ be as in the proof of
Theorem \ref{wiscs}.  We consider the following subalgebras of $B(H
\oplus K)$:
 $$  \left[
\begin{array}{ccl}
W^*_{\rm max}(M) & W^*_{\rm max}(M) X \\
Y W^*_{\rm max}(M) & Y W^*_{\rm max}(M) X
\end{array}
\right]  \; \; \; \; ,  \; \; \; \;  \left[
\begin{array}{ccl}
X W^*_{\rm max}(N) Y & X W^*_{\rm max}(N)  \\
 W^*_{\rm max}(N) Y &  W^*_{\rm max}(N)
\end{array}
\right] \; .
   $$
Let ${\mathcal L}_1$ and ${\mathcal L}_2$ denote the weak* closures
of these two subalgebras.   These are dual operator algebras which
are the linking algebras for a Morita equivalence of the type in the
present paper. Thus by Theorem \ref{alsodin}, they are actually
selfadjoint. Moreover both of these can now be seen to equal the von
Neumann algebra generated by ${\mathcal L}^w$, and so they are equal
to each other.    Now it is clear that, for example, the weak*
closures of $Y W^*_{\rm max}(M)$ and $W^*_{\rm max}(N) Y$ coincide,
and this constitutes an equivalence bimodule (or WTRO) setting up a
$W^{*}$-algebraic Morita equivalence between $W^{*}_{\rm max}(M)$
and $W^{*}_{\rm max}(N)$. The $W^{*}$-algebraic linking algebra here
is just ${\mathcal L}_1 = {\mathcal L}_2$, and this clearly contains
the algebra we called ${\mathcal R}$ in the discussion in the
beginning of Section 4, that is, ${\mathcal L}^w$, as a subalgebra.

Finally, notice that the map $\phi$ in the proof of the last theorem
is a completely isometric $W^*_{\rm max}(M)$-module map from $Z =  Y
\otimes^{\sigma h}_M \, W^*_{\rm max}(M)$ onto the weak* closure $W$
of $Y W^*_{\rm max}(M)$ in $B(H,K)$. Similar considerations, or
symmetry, shows that $V = W^*_{\rm max}(N) \otimes^{\sigma h}_N Y$
agrees with the weak* closure of $W^*_{\rm max}(N) Y$, which by the
above equals $W$, and thus agrees with $Z$.  Similarly for the
modules involving $X$.
\end{proof}

{\bf Remark.} \  Theorems 4 and 5 of \cite{BMN} have obvious variants
valid in our setting, with arbitrary $W^*$-dilations in place
of $W^*_{\rm max}(M)$.   Similarly, one can show as in
\cite{BMN} that $W^*_{\rm max}({\mathcal
L}^w) = {\mathcal L}_1$.  See \cite{UK} for details.

\medskip

{\bf Acknowledgements.}  The present paper is a second version of a 18 page
preprint distributed
on September 7, 2007, which only discussed `weak Morita equivalence'.   
Some days after this, we realized that all of the results 
and nearly all of the proofs worked
 in the more general setting of weak* Morita equivalence, 
and an update was immediately distributed informally.       The present version
(September 24, 2007) is a merging of the original paper and this update.

Just before submitting this revision, Vern Paulsen suggested we look again at
Example 3.7 in \cite{Enew} (which was previously part of \cite{Elef1}), and 
indeed this is clearly a weak Morita equivalence (see Example (10) in Section 3)
which is not a weak* stable isomorphism (by  \cite{Enew,EP}).   Thus Eleftherakis'
`Delta-equivalence' and the relations considered in our paper are distinct;
 each seem to have their own distinct
 advantages (see for
example the discussion on the first page of our paper).

\end{document}